\documentclass[11pt]{amsart}
\usepackage{graphicx, amsfonts, amssymb}
\usepackage{mathrsfs, amsmath, amsthm}
\usepackage{verbatim}
\usepackage{relsize}
\usepackage{epsfig}
\usepackage{color, colortbl}
\usepackage{comment}
\usepackage{mymacros}
\usepackage{xfrac,bigints}
\usepackage{enumerate}
\usepackage{hyperref}
\hypersetup{
	colorlinks=true,
	linkcolor=blue,
	filecolor=magenta,      
	urlcolor=cyan,
	pdftitle={Optimal Domain for Tg Operator},
	pdfpagemode=FullScreen,
}

\theoremstyle{definition}

 \theoremstyle{remark}

\numberwithin{equation}{section}

\def\BMOA{\text{BMOA}}

\begin{document}

\title[Hilbert matrix
between conformally invariant spaces]{Hilbert matrix operator Acting
between conformally invariant spaces}
\author{C. Bellavita}
\email{carlobellavita@ub.edu}
\address{Departament of Matem\'atica i Inform\'atica, Universitat de Barcelona, Gran Via 585, 08007 Barcelona, Spain.}

\author{G. Stylogiannis}
\email{stylog@math.auth.gr}
\address{Department of Mathematics, Aristotle University of Thessaloniki, 54124, Thessaloniki, Greece.}

\thanks{The first author is a member of Gruppo Nazionale per l’Analisi Matematica, la Probabilit\`a e le loro Applicazioni (GNAMPA) of Istituto Nazionale di Alta Matematica (INdAM) and he was partially supported by PID2021-123405NB-I00 by the Ministerio de Ciencia e Innovaci\'on and by the Departament de Recerca i Universitats, grant 2021 SGR 00087.\\
The second author was partially supported by the Hellenic Foundation for Research and Innovation (H.F.R.I.) under the '2nd Call for H.F.R.I. Research Projects to support Faculty Members \& Researchers' (Project Number: 4662).}



\begin{abstract}
In this article we study the action of the the Hilbert matrix operator $\mathcal H$ from the space of bounded analytic functions into conformally invariant Banach spaces. In particular, we describe the norm of $\mathcal{H}$ from $H^\infty$ into $\text{BMOA}$ and we characterize the positive Borel measures
$\mu$ such that $\mathcal H$ is bounded from $H^\infty$ into the conformally invariant Dirichlet space $M(\mathcal{D}_\mu )$. For particular measures $\mu$, we also provide the norm of $\mathcal{H}$ from $H^\infty$ into $M(\mathcal{D}_\mu )$.
\end{abstract}

\maketitle

\section{introduction}
\noindent Let $\mathbb{D}$ be the unit disc in the complex plane $\mathbb{C}$ and let $\text{Hol}(\mathbb{D})$ be the space of analytic functions in $\mathbb{D}$. With $\mathbb{T}$ we denote the unit circle, that is $\mathbb{T}=\partial \mathbb{D}$.
The  classical Hilbert matrix  is 
\[
\mathcal{H}=\left(
\begin{array}{ccccc}
            1 & \frac{1}{2}  & \frac{1}{3}  & \dots  \\ [4pt]
  \frac{1}{2} & \frac{1}{3}  & \frac{1}{4}  & \dots  \\ [4pt]
  \frac{1}{3} & \frac{1}{4}  & \frac{1}{5}  & \dots \\
  \vdots  & \vdots & \vdots & \ddots\\
\end{array}
\right).
\]

\noindent 
The matrix $\mathcal{H}$ introduces an operator on spaces of analytic functions through its action on the sequence of Taylor coefficients. For $$
f(z)=\sum_{m=0}^{\infty}a_{m}z^{m}
$$
in the Hardy space $H^1$, $\mathcal{H}(f)$ is defined as
$$
\mathcal{H}(f)(z)=\sum_{n=0}^{\infty}\left(\sum_{k=0}^{\infty}\frac{a_{k}}{k+n+1}\right)z^{n}.
$$
It is clear that $\mathcal{H}(f)$ is an analytic function on the unit
disk, since Hardy's inequality \cite[p.~48]{Duren70}
$$
\sum_{m=0}^{\infty}\frac{|a_m|}{m+1}\leq \pi \|f\|_{H^{1}}
$$
implies that the Taylor coefficients of $\mathcal{H}(f)$ are bounded. 

\vspace{11 pt}\noindent In \cite{Diamantopoulos2000}, E. Diamantopoulos and A. Siskakis initiated the study of the Hilbert
matrix as an operator on Hardy spaces by using the fact that for every $ z\in\mathbb{D}$, the function ${\mathcal{H}}(f)$ has the equivalent integral representation
\begin{equation}
\label{Integral representation}
    \mathcal{H}(f)(z)=\int_{0}^{1}\frac{f(t)}{1-tz}dt.
\end{equation}
By considering $\mathcal{H}$ as the average of weighted composition operators, they proved that $\mathcal{H}$ is a bounded operator on every Hardy space $H^p$ with $p > 1$ and they estimated its norm. Their research was further extended by M. Dostanic, M. Jevtic and D. Vukotic in \cite{Dostanic2008NormOT}.
Among other results, they showed that
$$
\|\mathcal{H}\|_{H^{p}\to H^p}=\frac{\pi}{\sin(\frac{\pi}{p})},
$$
for $1<p<\infty$. 

\vspace{11 pt}\noindent 
The study of the Hilbert matrix operator was subsequently extended to include the Bergman spaces of the unit disc. Diamantopoulos \cite{Diamantopoulos2004},  Dostanic et al. \cite{Dostanic2008NormOT} and V. Bozin and B. Karapetrovic \cite{Boin2018} proved that $\mathcal H$ is bounded on the Bergaman space $A^p$ if and only if $p>2$. Moreover 
$$
\|\mathcal{H}\|_{A^p\to A^p}= \frac{\pi}{\sin\left(\frac{2\pi}{p}\right) }.
$$

\noindent
Following the classical case, a significant body of research has focused on generalizations of the Hilbert matrix operator, particularly regarding the characterization of their boundedness and compactness, see, for example, \cite{Bellavita2024}, \cite{bellavitaF} and \cite{Bellavitasurvey} for a recent survey on the argument.

\vspace{11 pt}
\noindent 
In this article, we study the action of $\mathcal{H}$ from $H^\infty$ into conformally invariant Banach spaces. We recall that a Banach space $X$ of analytic functions is conformally invariant if for every $f \in X$ and $\phi \in \text{Aut}(\mathbb{D})$
$$
\|f\|_{X}=\|f\circ \phi\|_{X},
$$
where the Mobius group $\text{Aut}(\mathbb{D})$ is the set made by all the one-to-one analytic functions mapping $\mathbb{D}$ onto itself, see \cite{Arzay1985}.
Main examples of conformaly invariant spaces are the BMOA space, the $Q_p$ spaces and the classical Dirichlet space $\mathcal{D}$, while $H^p$ with $1\leq p<\infty$ is not.

\vspace{11 pt}\noindent 
In this article, we consider the action of the Hilbert matrix operator from $H^\infty$ into the space of analytic bounded mean oscillation $\BMOA$ and the conformally invariant Dirichlet spaces $M(\mathcal{D}_\mu)$. 

\vspace{11 pt}\noindent 
The $\BMOA$ space consists of all the functions in the Hardy space $H^2$ such that
\begin{align*}
\|f\|_{\BMOA}=& |f(0)|+\left(\sup_{a\in\mathbb{D}}\int_{\mathbb{D}}\left|f'(z)\right|^{2}\log\left|\frac{1-\overline{a}z}{a-z}\right|^{2}dA(z)\right)^{1/2}<\infty,
\end{align*}
where, if $z=x+iy$, $dA(z)=\pi^{-1}dxdy$.
With the above norm $\BMOA$ is a conformally invariant  Banach space and 
$$  H^\infty\subsetneq \BMOA \subsetneq  \bigcap_{0<p<\infty}H^p. $$
The norm $\|f\|_{\BMOA}$ can be expressed by integration on $\mathbb{T}$,
\begin{equation}
    \label{equivalent BMOA norm}
\|f\|_{\BMOA}= |f(0)|+\left( \sup_{a\in\mathbb{D}}\int_{\mathbb{T}}|f(e^{i\theta})|^2\frac{1-|a|}{|e^{i\theta}-a|^2}\frac{d\theta}{2\pi}-|f(a)|^2\right)^{1/2}.
\end{equation}
For more information about $\BMOA$, we refer to \cite[Chapter VI]{garnett2006} and \cite{Girela2001}.

\vspace{11 pt}

\begin{thm}\label{T:norm h h to bmoa}
The Hilbert matrix operator maps $H^\infty$ into $\BMOA$ and its norm is $1+\dfrac{\pi}{\sqrt{2}}$.
\end{thm}

\noindent The boundedness of $\mathcal{H}$ from $H^\infty$ into $\BMOA$ has been first observed by B. Lanucha, M. Nowak and M. Pavlovic in \cite{Lanucha2012}. Actually, it is also true that 
$$
\mathcal{H}\left( H^\infty\right)\subseteq \bigcap_{1<p<\infty}\Lambda\left(p,\frac{1}{p}\right)\subset \BMOA,
$$
where $\Lambda(p,\frac{1}{p})$ are the mean Lipschitz spaces (see later for definition).

\vspace{11 pt}\noindent 
Subsequently, we fix our attention on the conformally invariant Dirichlet space $M(\mathcal{D}_\mu )$.
Let $d\mu(z)$ be a positive, Borel measure in $\mathbb{D}$. The spaces $M(\mathcal{D}_\mu)$ consists of all the functions $f \in \text{Hol}(\mathbb{D})$ such that 
\begin{equation}\label{second normmdm2}
     \|f\|_{M(\mathcal{D}_{\mu})}=|f(0)|+\sup_{\phi \in \text{Aut}(\mathbb{D})}\left(\int_{\mathbb{D}} |f'(w)|^{2}U_{\mu}(\phi(w))dA(w)\right)^{1/2},
\end{equation}
where
$$
U_{\mu}(z) = \int_{\mathbb{D}}\log\left| \frac{1-\overline{w}z}{z-w}\right|^{2}d\mu(w) 
$$
is a superharmonic function in $\mathbb{D}$.
The most famous examples  of $M(\mathcal{D}_\mu)$ spaces are the $Q_p$ spaces, see \eqref{norm Q_p} and \cite{AULASKARI1995}. An equivalent expression for $\|f\|_{M(\mathcal{D}_\mu)}$ is

\begin{equation}
\label{second norm Mdm}
\|f\|_{M(\mathcal{D}_\mu)} \sim |f(0)|+ \sup_{\phi \in \text{Aut}(\mathbb{D})}\left( \int_{\mathbb{D}} |f'(w)|^{2}V_{\mu}(\phi(w))dA(w)\right)^{1/2},
\end{equation}
where
$$
 V_{\mu}(z) = \int_{\mathbb{D}}(1-|\sigma_{z}(w)|^{2})d\mu(w).
$$

\vspace{11 pt}\noindent
We characterize the measures $d\mu(z)$ such that the Hilbert matrix operator $\mathcal H$ is bounded from $H^\infty$ into $M(\mathcal{D}_\mu )$ . 

\begin{thm}\label{h from h to mdm}
Let $d\mu(z)$ be a positive Borel measure on $\mathbb{D}$. The following conditions are equivalent:
\begin{enumerate}[(i)]
\item The Hilbert matrix operator $\mathcal{H}$ sends $H^\infty$ into $M(\mathcal{D}_{\mu})$.
\item $\log(1-z) \in M(\mathcal{D}_{\mu})$.
\item 
\begin{equation*}\label{cond 1}
    \sup_{\lambda\in\mathbb{T}}\int_{\mathbb{D}}\frac{V_{\mu}(z)}{|1-\lambda z|^{2}}dA(z)<\infty.
\end{equation*}
\item
\begin{equation*}\label{cond 2}
    \sup_{a\in\mathbb{D}}\int_{\mathbb{D}}\frac{V_{\mu}(z)}{|1-az|^{2}}dA(z)<\infty.
\end{equation*}
\end{enumerate}
\end{thm}

\noindent We highlight that the proof of Theorem \ref{h from h to mdm} is similar to the proof of the analogous result for the Ces{\'a}ro operator, see \cite[Theorem 1.2]{Bao2019}. 
In addition, for some measures $d\mu(z)$, we also provide the norm of $\mathcal{H}$ from $H^\infty$ into $M(\mathcal{D}_\mu)$.

\begin{thm}\label{norm h from h to mdm}
Let $d\mu(z)$ be a positive, radial, Borel measure on $\mathbb{D}$. If $\mathcal{H}:H^\infty\to M(\mathcal{D}_\mu)$ is bounded then 
\begin{equation*}\label{Norm MDm}
   \|\mathcal{H}\|_{H^\infty\to M(\mathcal{D}_\mu)}=1+\left(\int_{\mathbb{D}}\frac{4}{|1-z^{2}|^{2}}U_{\mu}(z)dA(z)\right)^{1/2}.
\end{equation*}
\end{thm}
\noindent Significantly, the measure associated to the $Q_p$ spaces satisfies the hypothesis of the above theorem and we are able to compute the norm of the Hilbert matrix operator from $H^\infty$ into $Q_p$.

\vspace{11 pt}\noindent 
The methodology developed in this article also works  for the  Ces\'aro operator

$$
\mathcal{C}(f)(z) = \sum_{n=0}^{\infty} \frac{1}{n+1}\left(\sum_{k=0}^{n}a_{k}\right)z^{n}, \text{ with }\,f (z) = \sum_{n=0}^{\infty} a_{n}z^{n} \in \text{Hol}(\mathbb{D}).
$$
\noindent 
The analogy between the Hilbert matrix operator and the Ces\'aro operator comes from their matrix representations, that is  
\[
    \mathcal{C}=\left(
\begin{array}{ccccc}
            1 & 0  & 0 & 0 & \dots  \\ [4pt]
  \frac{1}{2} & \frac{1}{2}  & 0 & 0& \dots  \\ [4pt]
  \frac{1}{3} & \frac{1}{3}  & \frac{1}{3} & 0 &\dots  \\ [4pt]
  \vdots  & \vdots & \vdots & \vdots & \ddots\\
\end{array}
\right).
\]
We observe that $\mathcal{H}$ is obtained from $\mathcal{C}$ by the following formal manipulation: we erase the zeros in each column of $\mathcal{C}$ and we shift up the columns to their first non-zero entry. In rigorous terms this is equivalent to the following algebraic relation. Let
$e_n(z) = z^n, n = 0, 1, 2, ... $ be the monomials, which form an orthonormal basis of
$H^2$. We have that
$$
\mathcal{C}(e_{n})=S^{n}\mathcal{H}(e_{n}),
$$
where $S$ is the shift operator, that is $Sf(z)=zf(z)$.

\begin{thm}\label{norm C from h to mdm}
Let $d\mu(z)$ be a positive, radial, Borel measure on $\mathbb{D}$. If $\mathcal{C}:H^\infty\to M(\mathcal{D}_\mu)$ is bounded then 
\begin{equation*}
   \|\mathcal{C}\|_{H^\infty\to M(\mathcal{D}_\mu)}=1+\left(\int_{\mathbb{D}}\frac{4}{|1-z^{2}|^{2}}U_{\mu}(z)dA(z)\right)^{1/2}.
\end{equation*}
\end{thm}

\noindent In light of the analogy between $\mathcal{C}$ and $\mathcal{H}$, a comparative analysis of the Hilbert matrix operator results with their Ces\'aro operator counterparts is warranted.

\vspace{11 pt}
\noindent
The rest of the article is divided in six sections. Section 2 is devoted to preliminary material: we recall the definition of Hardy spaces and we describe the conformally invariant Dirichlet spaces $M(\mathcal{D}_\mu)$. In Section 3, we deal with $\BMOA$ and we prove Theorem \ref{T:norm h h to bmoa}. In section 4, we briefly describe the action of the Hilbert matrix operator from $H^\infty$ into the mean Lipschitz spaces. In section 5, we describe the action of the Hilbert matrix operator from $H^\infty$ into $M(\mathcal{D}_\mu)$ and we prove Theorem \ref{h from h to mdm}.\\
We prove Theorem \ref{norm C from h to mdm} in Section 7.
Finally, in section 6, we prove Theorem \ref{norm h from h to mdm}. 
We conclude the article with some open problems.

\vspace{11 pt}
\noindent 
We use the following notation. By the expressions $f \lesssim g$, we mean that there exists a positive constant $C$ such that  
$$
f\leq C g. 
$$
If both \(f\lesssim g\) and \(f\gtrsim g\) hold, we write \(f\sim g\). The capital letter \(C\), we denote constants whose values may change every time they appear. Finally with $\delta_{m,n}$ we denote the classical Kronecker symbol.
%
\vspace{22 pt}
\section{Preliminary}
\noindent
In this preliminary section, we recall some definitions  that will be used throughout the whole article.

\subsection{Hardy spaces}
Let   $1\leq p<\infty$ and   $f\in \text{Hol}(\mathbb{D})$. For  $0\leq r<1$, let
\[
M_p(r,f)=\left(\dfrac{1}{2\pi}\int_{0}^{2\pi}\vert f(re^{i\theta})\vert^p\, d\theta\right)^{\frac{1}{p}}\ 
\]
be the usual integral means of $f$ on the circle of radius $r$.  The Hardy space $H^p =H^p(\mathbb{D})$ consists of all the functions $f\in \text{Hol}(\mathbb D)$ such that 
\[
\|f\|_{H^p}=\sup_{0\leq r<1}M_p(r,f) <\infty\ 
\]
and, for $p=\infty$, $\,H^\infty$ consists of the bounded analytic functions on $\mathbb{D}$, i.e. all the functions in $\text{Hol}(\mathbb D)$ such that
\[
\|f\|_{\infty}:=\sup_{z\in\mathbb{D}}|f(z)|<\infty.
\]
See \cite{Duren70} for the theory of Hardy spaces.
Important examples of $H^\infty$ functions are the automorphisms, $\phi \in \text{Aut}(\mathbb{D})$.  We recall that every $\phi \in \text{Aut}(\mathbb{D})$ can be written as
$$
\phi(z) = e^{i\theta} \sigma_{a}(z) \text{ with } \sigma_{a}(z) = \frac{a-z}{1-\overline{a}z},
$$ 
where $\theta$ is real and $a\in\mathbb{D}$.

\subsection{Conformally invariant Dirichlet spaces}
The M\"obius invariant spaces  $M(\mathcal{D}_\mu)$ generated by the Dirichlet space $\mathcal{D}_\mu$ are defined in \eqref{second normmdm2} \footnote{ For the theory of superharmonic weighted Dirichlet spaces $\mathcal{D}_\mu$, we refer to \cite{Aleman1993}.}.
In order to avoid that $M(\mathcal{D}_\mu)$ contains only constant functions, we always assume that 
\begin{equation}\label{Def measure}
    \int_{\mathbb{D}}(1-|z|^{2})d\mu(z) <\infty .
\end{equation}
If \eqref{Def measure} does not hold, $M(\mathcal{D}_\mu)$ is called \emph{trivial}.
From \cite{BAO2018a}, it is known that if $M(\mathcal{D}_\mu)$ is not trivial, then $\mathcal{D} \subseteq M(\mathcal{D}_\mu) \subseteq \BMOA$. Furthermore,
$M(\mathcal{D}_\mu) = \BMOA$ if and only if $\mu(\mathbb{D}) < \infty$.

\vspace{11 pt}\noindent 
 For $0 < p < \infty$, the $Q_p$ space consists of all the functions $f \in \text{Hol}(\mathbb{D})$ such that
\begin{equation}\label{norm Q_p}
\|f\|_{Q_{p}}^{2} = |f(0)|+ \sup_{a\in\mathbb{D}}\left( \int_{\mathbb{D}} |f'(z)|^{2}(1 - |\sigma_{a}(z)|^{2})^{p} dA(z)\right)^{1/2} < \infty.
\end{equation}
Clearly, $Q_{p_{1}} \subseteq Q_{p_{2}}$ for $0 < p_1 < p_2 < \infty$
and for $0<p<1$  
$$
Q_p=M(\mathcal{D}_{\mu_p}),
$$
where 
\begin{equation}
\label{mdp}
d\mu_p(z) = -\Delta[(1-|z|^{2})^{p}]dA(z),
\end{equation}
and $\Delta$ is the Laplacian \cite{BAO2018a}.
We have that $Q_1 = \text{BMOA}$ and $Q_p \subsetneq \text{BMOA}$ when $0 < p < 1$. We refer to J. Xiao’s monographs \cite{Xiao2000} and \cite{Xiao2006}  for more results on $Q_p$ spaces.
%
\vspace{22 pt}
\section{Proof of Theorem \ref{T:norm h h to bmoa}}
\noindent 
It is well know that $\mathcal{H}$ does not map $H^\infty$ into its self. Indeed, by \eqref{Integral representation}, we note that 
\begin{align}\label{H1}
 \mathcal{H}(1)(z)&=\int_{0}^{1}\frac{1}{1-tz}dt=\frac{1}{z}\log\left(\frac{1}{1-z}\right)
\end{align}
and this function does not belong to $H^\infty$.

\begin{proof}[Proof of Theorem \ref{T:norm h h to bmoa}]
In \cite{Lanucha2012} Lanucha, Nowak and Pavlovic  proved that $\mathcal{H}$ is bounded from $H^\infty$ into $\BMOA$. 
By using the same computations of N. Danikas and Siskakis \cite[Theorem 1]{Danikas1993} and the expression \eqref{equivalent BMOA norm} for the $\BMOA$ norm, we realize that
\begin{equation*}\label{norm H1} 
\|\mathcal{H}(1)\|_{\BMOA}=1+\frac{\pi}{\sqrt{2}},
\end{equation*}
from which it follows that
\begin{equation}\label{estimate below}
1+\frac{\pi}{\sqrt{2}}\leq \|\mathcal{H}\|_{H^\infty\to \BMOA}.
\end{equation}
\noindent 
In order to prove the upper bound for the norm of $\mathcal{H}$, because of \eqref{Integral representation}, we note that 
\begin{equation*}\label{derivative Hilbert transform}
 \mathcal{H} (f)'(z)=\int_{0}^{1}\frac{t\,f(t)}{(1-tz)^{2}}dt.
\end{equation*}
The convergence of the integral and the analyticity of the function $f$ guarantee that we can change the path of integration to 
$$
z(t) = \frac{s}{(s-1)w+1} ,\qquad  0 \leq s \leq 1,
$$
which describes circular arcs contained in $\mathbb{D}$. Therefore, we obtain that
\begin{equation*}\label{main identity}
    \mathcal{H}(f)'(w)=\frac{1}{1-w}\int_{0}^{1}\frac{s}{(s-1)w+1}f\left(\frac{s}{(s-1)w+1}\right)ds.
\end{equation*}
The quantity inside the integral is $\psi_{s}(w)f(\psi_{s}(w))$, where  
$$
\psi_{s}(w)=\frac{s}{(s-1)w+1}
$$
 maps the unit disc into itself for each $0\leq s<1$.
 Thus, in absolute value, 
 $$
 |\psi_{s}(w)f(\psi_{t}(w))|\leq \|f\|_{H^{\infty}}
 $$ 
 for each $w \in \mathbb{D}$ and $0\leq s\leq 1$. Consequently,
\begin{equation*}
\begin{split}
\|\mathcal{H}(f)\|_{\BMOA}&= |\mathcal{H}(f)(0)|+\sup_{a\in\mathbb{D}}\left(\int_{\mathbb{D}}\left|\mathcal{H}(f)'(z)\right|^{2}\log\left|\frac{1-\overline{a}z}{a-z}\right|^{2}dA(z)\right)^{1/2}\\
&\leq \left(1+\sup_{a\in\mathbb{D}}\left(\int_{\mathbb{D}}\left|\frac{1}{1-z}\right|^{2}\log\left|\frac{1-\overline{a}z}{a-z}\right|^{2}dA(z)\right)^{1/2}\right)\|f\|_{H^{\infty}}\\
&=\left(1+\frac{\pi}{\sqrt{2}}\right)\|f\|_{H^{\infty}},
\end{split}
\end{equation*}
where in the last equality we have used the fact that $\|\log(1-z)\|_{\BMOA}=\frac{\pi}{\sqrt{2}}$, see \cite{Danikas1984}. Due to the above inequality and \eqref{estimate below}, we have the conclusion.\\
\end{proof}

\noindent 
The proof of Theorem \ref{T:norm h h to bmoa} provides a usefully identity for the derivative $\mathcal{H}(f)'$ when $f\in H^\infty$. Indeed 
\begin{equation}\label{Main id for Deriv}
  \mathcal{H}(f)'(z)=\frac{b(z)}{1-z},
\end{equation}
with $b\in H^\infty$ and $\|b\|_{H^{\infty}}\leq\|f\|_{H^{\infty}}$.   
%
\vspace{22 pt}
\section{Hilbert matrix operator into mean Lipschitz spaces}
\noindent 
The inclusion $\mathcal{H}\left( H^\infty\right)\subset \BMOA$ can be refined to
$$
\mathcal{H}\left( H^\infty\right)\subseteq \bigcap_{1<p<\infty}\Lambda\left(p,\frac{1}{p}\right).
$$ 
 For $1 < p < \infty$ and $0 < \alpha \leq 1$ the mean Lipschitz space $\Lambda(p,\alpha)$ consists  all the analytic functions $f$ on the unit disc for which
$$
\|f\|_{\Lambda(p,\alpha)}=|f(0)|+\sup_{0<r<1}M_p(r,f')(1-r^{2})^{1-\alpha}<\infty.
$$
The spaces $\Lambda(p,\frac{1}{p})$ with $ 1 < p < \infty$ grow in size with $p$, they are all subspaces of $\BMOA$ and they all contain unbounded functions such as $\log(1-z)$, see \cite{Bourdon1989}.

\vspace{11 pt}\noindent 

\begin{thm}\label{Prop norm hilbert in Lipschitz}
The Hilbert matrix operator $\mathcal{H}$ is bounded from $H^\infty$ into $\Lambda(p,1/p)$ and
$$
\|\mathcal{H}\|_{H^\infty\to \Lambda(p,\frac{1}{p})}\leq 1+\|\log(1-z) \|_{\Lambda(p,\frac{1}{p})}.
$$
\end{thm}
\begin{proof}
The proof is similar to Theorem \ref{T:norm h h to bmoa}.  By \cite[Theorem 1.7]{Hedenmalm2000}, we have that $\log(1-z) \in \Lambda(p,\frac{1}{p}) $ for $1<p<\infty$ with 
$$
\|\log(1-z)\|_{\Lambda(p,\frac{1}{p})}=\sup_{0<r<1}
(1-r)^{1-1/p}\left( \int_{\mathbb{T}} \left| \frac{1}{1-re^{-i\theta}}\right|^p d\theta\right)^{1/p}.
$$
By \eqref{Main id for Deriv}, we have that
\begin{align*}
   (1-r^2)^{1-1/p}M_p(r,\mathcal{H}(f)')&\leq (1-r^2)^{1-1/p}\left( \int_{\mathbb{T}} \left| \frac{1}{1-re^{-i\theta}}\right|^p d\theta\right)^{1/p}\|f\|_{H^\infty} \\
    &\leq \|\log(1-z)\|_{\Lambda(p,\frac{1}{p})}\|f\|_{H^\infty}.
\end{align*}
Consequently
$$
\|\mathcal{H}(f)\|_{\Lambda(p,\frac{1}{p})}\leq \left( 1+\|\log(1-z)\|_{\Lambda(p,\frac{1}{p})}\right)\|f\|_{H^\infty}.
$$
\end{proof}
\noindent
We point out that the exact value of the $\|\mathcal{H}\|_{H^\infty\to \Lambda(p,\frac{1}{p})}$
is not provided by Theorem \ref{Prop norm hilbert in Lipschitz}. However, we are able to compute it when $p=2$.
\begin{prop}
The Hilbert matrix operator $\mathcal{H}$ is bounded from $H^\infty$ into $\Lambda(2,1/2)$ and
$$
\|\mathcal{H}\|_{H^\infty\to \Lambda(2,\frac{1}{2})}= 1+\|\log(1-z) \|_{\Lambda(2,\frac{1}{2})}=2.
$$
\end{prop}
\begin{proof}
Since in Theorem \ref{Prop norm hilbert in Lipschitz} we have already shown the upper bound, we have to provide only the estimate from below.    
Since $\mathcal{H}(1)(1)=1$ and 
$$
S\,(\mathcal{H}(1)'(z))=-\frac{1}{z}\log\left(\frac{1}{1-z}\right)+\frac{1}{1-z}=f_{1}(z)+f_{2}(z),
$$
we have 
\begin{align*}
\|\mathcal{H}\|_{H^\infty \to \Lambda(2,\frac{1}{2})}&\geq \|\mathcal{H}(1)\|_{\Lambda(2,\frac{1}{2})}\\
&=1+\sup_{0<r<1}M_2(r,\mathcal{H}(1)')(1-r^{2})^{1/2}\\
&\geq 1+ \lim_{r\to1}M_2(r,S\,\mathcal{H}(1)')(1-r^{2})^{1/2}\\
&\geq 1+\lim_{r\to1}|M_2(r,f_1)-M_2(r,f_2)|(1-r^{2})^{1/2}\\
&=1+\lim_{r\to1}M_{2}(r, \left(\log(1-z)\right)')(1-r^{2})^{1/2}=1+\|\log(1-z)\|_{\Lambda(2,\frac{1}{2})}.
\end{align*}
Since $\|\log(1-z)\|_{\Lambda(2,\frac{1}{2})}=1$, conclusion follows.\\
\end{proof}
\vspace{22 pt}
\section{Hilbert matrix operator into conformally invariant Dirichlet spaces}
\noindent 
We start by proving Theorem \ref{h from h to mdm}.
\begin{proof}[Proof of Theorem \ref{h from h to mdm}]
The proof of Theorem \ref{h from h to mdm} is similar to \cite[Theorem 1.2]{Bao2019}. Nevertheless, for completeness, we include it here.

\vspace{11 pt}\noindent 
$(i)\Rightarrow(ii)$. Let $\mathcal{H}\left(H^\infty\right) \subseteq M(\mathcal{D}_{\mu})$.  Since 
$ \mathcal{H}(1) \in \mathcal{H}\left(H^\infty\right)$, 
we have that $M(\mathcal{D}_{\mu})$ is not trivial and \cite[Theorem 3.3]{BAO2018a} implies that 
$$
\sup_{w\in\mathbb{D}} V_{\mu}(w) < \infty.
$$
In order to prove that $\log(1-z)$ is in $M(D_{\mu})$, 
according to \eqref{second norm Mdm}, it is enough to show that
\begin{equation}\label{itoii}
\sup_{\phi \in \text{Aut}(\mathbb{D})}\int_{\mathbb{D}\setminus \frac{1}{2}\mathbb{D}}\left|\frac{1}{1-z}\right|^{2}V_{\mu}(\phi(z))dA(z)<\infty.
\end{equation}
Indeed
\begin{align*}
\infty&>\sup_{\phi \in \text{Aut}(\mathbb{D})}\int_{\mathbb{D}\setminus \frac{1}{2}\mathbb{D}}\left|\mathcal{H}(1)'(z)\right|^{2}V_{\mu}(\phi(z))dA(z)\\
&=\sup_{\phi \in \text{Aut}(\mathbb{D})}\int_{\mathbb{D}\setminus \frac{1}{2}\mathbb{D}}\left|\frac{1}{z}\frac{1}{1-z}-\frac{1}{z^{2}}\log\frac{1}{1-z}\right|^{2}V_{\mu}(\phi(z))dA(z).
\end{align*}
Since
\begin{align*}
\sup_{\phi \in \text{Aut}(\mathbb{D})}\int_{\mathbb{D}\setminus \frac{1}{2}\mathbb{D}}\left|\frac{1}{z^{2}}\log\frac{1}{1-z}\right|^{2}&V_{\mu}(\phi(z))dA(z)\leq\\
&16\sup_{w\in\mathbb{D}} V_{\mu}(w)\int_{\mathbb{D}\setminus \frac{1}{2}\mathbb{D}}\left|\log\frac{1}{1-z}\right|^{2}dA(z)<\infty,
\end{align*}
\eqref{itoii} holds.

\vspace{11 pt} 
\noindent $(ii)\Rightarrow(iii)$. Let $\log(1-z)\in M(\mathcal{D}_{\mu})$. By using \eqref{second norm Mdm} and a change of variables, we have that 
$$
\sup_{a\in\mathbb{D},\lambda\in\mathbb{T}}\int_{\mathbb{D}}\left|\frac{1}{1-\lambda\sigma_{a}(z)}\right|^{2}\frac{(1-|a|^{2})^{2}}{|1-\overline{a}z|^{4}}V_{\mu}(z)dA(z)<\infty.
$$
Taking $a=0$  in the above condition, we have 
$$
\sup_{\lambda\in\mathbb{T}}\int_{\mathbb{D}}\frac{V_{\mu}(z)}{\left|1-\lambda z\right|^{2}}dA(z)<\infty.
$$

\vspace{11 pt}
\noindent 
\noindent $(iii)\Leftrightarrow(iv)$. We set 
$(1-|z|^2)d\nu(z)=V_{\mu}(z)dA(z)$ and apply \cite[Lemma 2.2]{Bao2018}.

\vspace{11 pt}\noindent 
\noindent $(iv)\Rightarrow(i)$. 
Let $f \in H^\infty$. 
By using \eqref{Main id for Deriv}, we note that
\begin{align*}
  &\int_{\mathbb{D}} \left|\mathcal{H}(f)'(z)\right|^{2}V_{\mu}(\phi(z))dA(z)\\
  &\qquad\leq \int_{\mathbb{D}} \left|\frac{1}{1-z}\right|^{2}V_{\mu}(\phi(z))dA(z)\,\|f\|_{\infty}^2\\
  &\qquad =\int_{\mathbb{D}} \left|\frac{1}{1-\phi(z)}\right|^{2} \frac{\left(1-|a|^2\right)^2}{|1-\overline{a}z|^4}V_{\mu}(z)dA(z)\,\|f\|_{\infty}^2,
  \end{align*}
  where we have chosen 
  $$
  \phi(z)=\lambda \frac{a-z}{1-\overline{a}z},\,\, \lambda\in\mathbb{T},a\in\mathbb{D}.
  $$
Consequently, $\mathcal{H}$ is bounded from $H^\infty$ into $M(\mathcal{D}_\mu)$ if 
\begin{align*}\label{Cond for M(Dm) eq }
I=&\sup_{a\in\mathbb{D,\lambda\in\mathbb{T}}}\int_{\mathbb{D}}\left|\frac{1}{1-\lambda\sigma_{a}(z)}\right|^{2} \frac{(1-|a|^{2})^{2}}{|1- \overline{a}z|^{4}}V_{\mu}(z)dA(z)\\
=&\sup_{a\in\mathbb{D,\lambda\in\mathbb{T}}}\frac{(1-|a|^{2})^{2}}{|1- \lambda a|^{2}}\int_{\mathbb{D}}\left|\frac{1}{1+\frac{\lambda-\overline{a}}{1-\lambda a}z}\right|^{2} \frac{1}{|1- \overline{a}z|^{2}}V_{\mu}(z)dA(z)<\infty.
\end{align*}
We set $\eta=\frac{\lambda-\overline{a}}{1-\lambda a}$, $|\eta|=1$. Through a change of variables, we get 
\begin{align*}
I&=\sup_{a\in\mathbb{D,\lambda\in\mathbb{T}}}\frac{(1-|a|^{2})^{2}}{|1- \lambda a|^{2}}\int_{\mathbb{D}}\left|\frac{1}{1+\zeta}\right|^{2} \frac{1}{|1- \overline{a\eta}\zeta|^{2}}V_{\mu}(\overline{\eta}\zeta)dA(\zeta)\\
&=\sup_{a\in\mathbb{D,\lambda\in\mathbb{T}}}\int_{\mathbb{D}}\left|\frac{1}{1+\zeta}+\frac{\overline{a\eta}\zeta}{1- \overline{a\eta}\zeta}\right|^{2}V_{\mu}(\overline{\eta}\zeta)dA(\zeta)\\
&\leq C\left(\sup_{\eta\in\mathbb{T}}\int_{\mathbb{D}}\left|\frac{1}{1+\zeta}\right|^{2}V_{\mu}(\overline{\eta}\zeta)dA(\zeta)+\sup_{a\in\mathbb{D,\lambda\in\mathbb{T}}}\int_{\mathbb{D}}\left|\frac{1}{1- \overline{a\eta}\zeta}\right|^{2}V_{\mu}(\overline{\eta}\zeta)dA(\zeta)\right) \\
&\leq C \sup_{\eta\in\mathbb{T}}\int_{\mathbb{D}}\left|\frac{1}{1+\eta z}\right|^{2}V_{\mu}(z)dA(z)+\sup_{a\in\mathbb{D}}\int_{\mathbb{D}}\left|\frac{1}{1- \overline{a}z}\right|^{2}V_{\mu}(z)dA(z).
\end{align*}
By using conditions $(iii)$ and $(iv)$, the proof is complete.\\
\end{proof}

\begin{rem}\label{counter ex}
Let
$$
d\mu_{a}(z)=\Delta\left(\frac{1}{\left(\log\frac{e^{1+a}}{1-|z|^{2}}\right)^{a}}\right)dA(z)
$$
with $z \in \mathbb{D}$, $a \in (0,\infty)$ and $\Delta$ the classical Laplace operator. The proof \cite[Theorem 1.2]{Bao2019} gives that
\begin{enumerate}[(i)]
    \item If $0 < a \leq 1$, then $M(\mathcal{D}_{\mu_{a}} )$ is not trivial and $\mathcal{H}(H^\infty) \not\subseteq M(\mathcal{D}_{\mu_{a}} )$.
    \item If $a>1$, then $\mathcal{H}(H^\infty) \not\subseteq M(\mathcal{D}_{\mu_{a}} ) \not\subseteq \bigcap_{0<p<\infty} Q_p$.
\end{enumerate}
\noindent Consequently, there are measures $\mu$ for which $\mathcal{H}(H^\infty) \not\subseteq M(\mathcal{D}_\mu).$
\end{rem} 

\vspace{22 pt}
\section{Norm of the Hilbert matrix operator into conformally invariant Dirichlet spaces}
\noindent For particular measures $\mu$, it is possible to compute $\|\mathcal{H}\|_{H^\infty\to M(\mathcal{D}_\mu)}$. We need two preliminary lemmas.
\begin{lem}
Let $d\mu(z)$ be a radial measure. Then $U_\mu(z)dA(z)$ is also a radial measure.    
\end{lem}
\begin{proof}
Let $\lambda \in \mathbb{T}$. We note that
\begin{align*}
U_\mu(\lambda z)=&\int_\mathbb{D} \log \left\vert \frac{1-\overline{w}\lambda z}{\lambda z-w}\right\vert d\mu(w)=\int_\mathbb{D} \log \left\vert \frac{1-\overline{w\overline{\lambda}}z}{z-\overline{\lambda} w}\right\vert d\mu(w)\\
=&\int_\mathbb{D} \log \left\vert \frac{1-\overline{w\overline{\lambda}}z}{z-\overline{\lambda} w}\right\vert d\mu(\overline{\lambda}w)=U_\mu(z).
\end{align*}
Therefore
$$
U_\mu(z)dA(z)=U_\mu(r)\,rdr\frac{d\theta}{2\pi},
$$
from which the Lemma follows.\\
\end{proof}

\begin{lem}\label{estimates moment}
Let $d\mu(z)$ be a positive, radial measure such that $\log(1-z) \in M(\mathcal{D}_\mu)$. Then
$$
\sum_{n=1}^{\infty}  \widehat{U_\mu}_{2n+1}<\infty ,
$$
where
$$
\widehat{U_\mu}_{2n+1}=\int_0^1 r^{2n+1}U_\mu(r)dr .
$$
\end{lem}
\begin{proof}
We note that
\begin{align*}
\sum_{n=1}^\infty  \widehat{U_\mu}_{2n+1}=&\sum_{n=1}^\infty  \int_{0}^1 r^{2n}U_\mu(r)rdr=\sum_{n=0}^\infty  \int_{0}^1 r^{2n}U_\mu(r)rdr-\widehat{U_\mu}_{1}\\
\leq &  \int_{0}^1 \frac{1}{1-r^2}U_\mu(r)rdr=\int_{0}^1 \int_{0}^{2\pi} \frac{1}{1-e^{i\theta}r}\frac{1}{1-e^{-i\theta}r} \frac{d\theta}{2\pi} U_\mu(r)rdr\\
=&\int_{\mathbb{D}} \left\vert \frac{1}{1-z}\right\vert^2 U_\mu(z)dA(z)\leq\|\log(1-z)\|_{M(\mathcal{D}_\mu)}^2.
\end{align*}
Since $\log(1-z) \in M(\mathcal{D}_\mu)$, the conclusion follows.\\
\end{proof}

\noindent 
When $d\mu(z)$ is a radial measure, \eqref{second normmdm2} becomes
\begin{align*}    
 \|f\|_{M(\mathcal{D}_{\mu})}&
 =|f(0)|+\sup_{a \in \mathbb{D}, \lambda \in \mathbb{T}}\left(\int_{\mathbb{D}} |f'(w)|^{2}U_{\mu}(\lambda\sigma_a(w))dA(w)\right)^{1/2}\\
&=|f(0)|+\sup_{a \in \mathbb{D}}\left(\int_{\mathbb{D}} |f'(w)|^{2}U_{\mu}(\sigma_a(w))dA(w)\right)^{1/2}.
\end{align*}

\begin{prop}\label{norm log MDmu}
Let $d\mu(z)$ be a positive, radial, Borel measure on $\mathbb{D}$. 
Then 
\begin{equation*}
\|\log(1-z)\|^2_{M(\mathcal{D}_\mu)}=
\int_{\mathbb{D}}\left|\frac{2}{1-z^2}\right|^{2}U_{\mu}(z)dA(z).
\end{equation*}
\end{prop}
\begin{proof}
Since the measure $d\mu(z)$ is radial, we observe that 
\begin{align*}
\|\log(1-z)\|_{M(\mathcal{D}_\mu)}&=\sup_{a\in\mathbb{D}}\int_{\mathbb{D}}\left|\frac{1}{1-z}\right|^{2}U_{\mu}(\sigma_{a}(z))dA(z)\\
&=\sup_{a\in\mathbb{D}}\int_{\mathbb{D}}\left|\frac{1}{1-\sigma_{a}(z)}\right|^{2}|\sigma_{a}'(z)|^{2}U_{\mu}(z)dA(z)\\
&=\sup_{a\in\mathbb{D}}\int_{\mathbb{D}}\frac{(1-|a|^{2})^{2}}{|1-a+(1-\overline{a})z|^{2}}\frac{1}{|1-\overline{a}z|^{2}}\,U_{\mu}(z)dA(z)\\
&=\sup_{a\in\mathbb{D}}\int_{\mathbb{D}}\left|\frac{1-\overline{a}}{1-a+(1-\overline{a})z}+\frac{\overline{a}}{1-\overline{a}z}\right|^{2}U_{\mu}(z)dA(z).
\end{align*}
In particular, by using the power series expansion, we have that
\begin{small}
\begin{align*}
\|\log(1-z) \|_{M(\mathcal{D}_\mu)}
&=\sup_{a\in\mathbb{D}}\int_{\mathbb{D}}\left|\sum_{n=0}^{\infty}\left((-1)^{n}\left(\frac{1-\overline{a}}{1-a}\right)^{n+1}+\overline{a}^{n+1}\right)z^{n}\right|^{2}U_{\mu}(z)dA(z)\\
&=\sup_{a\in\mathbb{D}}\int_{\mathbb{D}}\left|\sum_{n=0}^{\infty}\left(-\left(\frac{1-\overline{a}}{1-a}\right)^{n+1}+\overline{-a}^{n+1}\right)(-z)^{n}\right|^{2}U_{\mu}(z)dA(z)\\
&=\sup_{a\in\mathbb{D}}\int_{\mathbb{D}}\left|\sum_{n=1}^{\infty}\left(-\left(\frac{1-\overline{a}}{1-a}\right)^{n}+\overline{-a}^{n}\right)(-z)^{n-1}\right|^{2}U_{\mu}(z)dA(z).
\end{align*}
\end{small}
Since $d\mu(z)$ is radial and 
$$
\int_0^{2\pi} e^{(n-1)i\left(\theta+\pi\right)}e^{-(m-1)i\left(\theta+\pi\right)}\frac{d\theta}{2\pi}=\delta_{n,m},\quad n,m\in\mathbb{N},
$$
we have that
\begin{align*}
\|\log(1-z)\|_{M(\mathcal{D}_\mu)}&=\sup_{a\in\mathbb{D}} \int_0^1 \sum_{n=1}^\infty\left\vert-\left(\frac{1-\overline{a}}{1-a}\right)^{n}-\overline{a}^{n}\right\vert^2 r^{2n+1}U_\mu(r)dr\\
&= \sup_{a\in\mathbb{D}}  \sum_{n=1}^\infty\left\vert-\left(\frac{1-\overline{a}}{1-a}\right)^{n}-\overline{a}^{n}\right\vert^2 \widehat{U_\mu}_{2n+1}.
\end{align*}
Hence,
\begin{small}
\begin{align*}
\|\log(1-z) \|_{M(\mathcal{D}_\mu)}&=\sup_{a\in\mathbb{D}}\sum_{n=1}^\infty  \widehat{U_\mu}_{2n+1}\left( 1+|a|^{2n}+2(-1)^{n+1}\Re\left(\frac{1-\overline{a}}{1-a} a\right)^{n}\right) .
\end{align*}
\end{small}
With the change of variable
\begin{equation*}\label{change variable}
 w=\frac{1-\overline{a}}{1-a} a , 
\end{equation*}
where $|w|=|a|$ for $a \in \overline{\mathbb{D}}\setminus \{1\}$ and $|w|\to 1$ as $a\to 1$,
we note that the function
$$
B(w)= \sum_{n=1}^\infty \widehat{U_\mu}_{2n+1} 2(-1)^{n+1}w^n
$$
is analytic in the unit disc. Moreover, since $B(w)$ is continuous in $\overline{\mathbb{D}}$ due to Lemma \ref{estimates moment}, the maximum principle for harmonic functions tells us that
$$
\sup_{w \in \mathbb{ D}}\Re B(w)=\sup_{w \in \mathbb{T}}\Re B(w) .
$$
Consequently,
\begin{small}
\begin{align*}
\sup_{w \in \mathbb{T}}\Re B(w)&= -2\inf_{w \in \mathbb{T}}\Re\left( \sum_{n=1}^\infty \widehat{U_\mu}_{2n+1} (-1)^{n}w^n\right)\\
&=-2\inf_{w \in \mathbb{T}}\Re \int_{0}^1 \sum_{n=1}^\infty (-wr^2)^{n}U_\mu(r)rdr\\
&=2\widehat{U_\mu}_{1}-2\inf_{w \in \mathbb{T}}\Re \int_{0}^1  \frac{1}{1+wr^2}U_\mu(r)rdr\\
&=2\left( \widehat{U_\mu}_{1}-\int_{0}^1  \frac{1}{1+r^2}U_\mu(r)rdr\right)=\Re B(1) .
\end{align*}
\end{small}
In particular, due to Lemma \ref{estimates moment}, we know that
\begin{align*}
    \|\log(1-z) \|_{M(\mathcal{D}_\mu)}=&\lim_{a\to 1} \int_{\mathbb{D}}\left|\frac{1}{1-z}\right|^{2}U_{\mu}(\sigma_{a}(z))dA(z)\\
    =& \lim_{a\to 1, a\in \mathbb{R}} \int_{\mathbb{D}}\left|\frac{1}{1-z}\right|^{2}U_{\mu}(\sigma_{a}(z))dA(z)\, .
\end{align*}
Consequently
\begin{align*}
    \|\log(1-z) \|_{M(\mathcal{D}_\mu)}&=\lim_{a\to 1, a\in \mathbb{R}} \int_{\mathbb{D}}\left|\frac{1}{1-z}\right|^{2}U_{\mu}(\sigma_{a}(z))dA(z)\\
&=\lim_{a\to 1, a\in \mathbb{R}}  \int_{\mathbb{D}}\left|\frac{1}{1-\sigma_{a}(z)}\right|^{2}|\sigma_{a}'(z)|^{2}U_{\mu}(z)dA(z)\\
&=\lim_{a\to 1, a\in \mathbb{R}}  \int_{\mathbb{D}} \left\vert \frac{1+a}{(1+z)(1-az)}\right\vert^2 U_{\mu}(z)dA(z)\\
&=\int_{\mathbb{D}} \left\vert \frac{2}{1-z^2}\right\vert^2 U_{\mu}(z)dA(z)\, .
\end{align*}
\end{proof}

\noindent By using Proposition \ref{norm log MDmu}, we are now able to prove Theorem \ref{norm h from h to mdm}.
\begin{proof}[Proof of Theorem \ref{norm h from h to mdm}]
As we saw earlier \eqref{Main id for Deriv}, we have that 
$$
(1-z)\mathcal{H}(f)'(z)=b(z),
$$
where $b\in H^\infty$ and $\|b\|_{H^\infty}\leq \|f\|_{H^\infty}$. 
Consequently,
\begin{align*}
 \|\mathcal{H}(f)\|_{M(\mathcal{D}_{\mu})}&=|\mathcal{H}(f)(0)|+\sup_{\phi \in \text{Aut}(\mathbb{D})}\left(\int_{\mathbb{D}} |\mathcal{H}(f)'(w)|^{2}U_{\mu}(\phi(w))dA(w)\right)^{1/2}\\
 &=\left\vert\int_{0}^{1}f(t)dt\right\vert+\sup_{\phi \in \text{Aut}(\mathbb{D})}\left(\int_{\mathbb{D}} \left|\frac{b(w)}{1-w}\right|^{2}U_{\mu}(\phi(w))dA(w)\right)^{1/2}\\
 &\leq (1+\|\log(1-z) \|_{M(\mathcal{D}_\mu)})\|f\|_{H^\infty} .
\end{align*}
By Proposition \ref{norm log MDmu}, we have that
\begin{equation}\label{Norm M(Dm)}
  \|\mathcal{H}\|_{H^\infty\to M(\mathcal{D}_\mu)}\leq 1+\left(\int_{\mathbb{D}}\frac{4}{|1-z^{2}|^{2}}U_{\mu}(z)dA(z)\right)^{1/2}.
\end{equation}
Once again, we recall that 
\begin{equation}\label{defn f_1}
    z\left( \mathcal{H}(1)\right)'(z)=-\frac{1}{z}\log\left(\frac{1}{1-z}\right)+\frac{1}{1-z}=f_{1}(z)+f_{2}(z).
\end{equation}
We observe that the function
$h(z)=(1-z)f_1(z)$ is in $H^\infty$ with 
$$
\lim_{z\to 1}h(z)=0.
$$
Let $0<a<1$ and set
\begin{align*}
I^a= \int_{\mathbb{D}} |f_1(z)|^{2}U_{\mu}(\sigma_{a}(z))dA(z)&=\int_{\mathbb{D}} \left|\frac{h(z)}{(1-z)}\right|^{2}U_{\mu}(\sigma_{a}(z))dA(z)\\
&=\int_{\mathbb{D}} \left|\frac{h(\sigma_{a}(z))}{1-\sigma_{a}(z)}\right|^{2}|\sigma_{a}'(z)|^{2}U_{\mu}(z)dA(z)\\
&=\int_{\mathbb{D}} g_{a}(z)dA(z),
\end{align*}
where
$$
g_{a}(z)=\left|h(\sigma_{a}(z))\right|^{2}\frac{(1+a)^{2}}{|(1-az)(1+z)|^{2}}U_{\mu}(z).
$$
For $z\in \mathbb{D}$ fixed observe that $\lim_{a\to1}\sigma_{a}(z)=1$.
This implies  that for each fixed $z\in\mathbb{D}$, $g_a(z)\to0$ as $a\to1$.
Since $h\in H^\infty$ with norm $M$, we have 
$$
g_{a}(z)\leq k_{a}(z)=M^2 \frac{(1+a)^{2}}{|(1-az)(1+z)|^{2}}U_{\mu}(z).
$$
Moreover
$$
k_a(z)\to \frac{4M^2}{|1-z^2|^{2}}U_{\mu}(z),
$$
as $a\to1$ and, by repeating the proof of Proposition \ref{norm log MDmu}, 
$$
\int_{\mathbb{D}}k_a (z)dA(z)\to \int_{\mathbb{D}}\frac{4M^2}{|1-z^2|^{2}}U_{\mu}(z)dA(z)<\infty. 
$$
Therefore by dominated convergence \cite[p.59; Ex.20]{folland2013}, we have that 
\begin{equation}\label{Lim Ia}
    \lim_{ a \to 1} I^a=0.
\end{equation}
According to Proposition \ref{norm log MDmu}, we have that 
\begin{align*}
 \|\mathcal{H}(1)\|_{M(\mathcal{D}_{\mu})}&=|\mathcal{H}(1)(0)|+\sup_{\phi \in \text{Aut}(\mathbb{D})}\left(\int_{\mathbb{D}} |\mathcal{H}(1)'(w)|^{2}U_{\mu}(\phi(w))dA(w)\right)^{1/2}\\
 &\geq 1+\sup_{a\in\mathbb{D}}\left(\int_{\mathbb{D}} |z\mathcal{H}(1)'(w)|^{2}U_{\mu}(\sigma_{a}(w))dA(w)\right)^{1/2}\\
 &\geq 1+\lim_{\mathbb{R}\ni a\to1}\left(\int_{\mathbb{D}} |f_{1}(w)+f_{2}(w)|^{2}U_{\mu}(\sigma_{a}(w))dA(w)\right)^{1/2}.
 \end{align*}
 By using \eqref{Lim Ia}, the last $\liminf$ becomes 
\begin{align*}
 \lim_{\mathbb{R}\ni a\to1}&\left(\int_{\mathbb{D}} |f_{1}(w)+f_{2}(w)|^{2}U_{\mu}(\sigma_{a}(w))dA(w)\right)^{1/2}\\
 &\geq \lim_{\mathbb{R}\ni a\to1}  \left|\,\left(\int_{\mathbb{D}} |f_{1}(w)|^{2}U_{\mu}(\sigma_{a}(w))dA(w)\right)^{1/2}-\left(\int_{\mathbb{D}} |f_{2}(w)|^{2}U_{\mu}(\sigma_{a}(w))dA(w)\right)^{1/2}\right|\\
 &=\lim_{\mathbb{R}\ni a\to1}\left(\int_{\mathbb{D}} \left|\frac{1}{1-w}\right|^{2}U_{\mu}(\sigma_{a}(w))dA(w)\right)^{1/2}.\\
 \end{align*}
This implies that 
 \begin{align*}
\|\mathcal{H}(1)\|_{M(\mathcal{D}_{\mu})}     &\geq 1+\lim_{\mathbb{R}\ni a\to1}\left(\int_{\mathbb{D}} \left|\frac{1}{1-w}\right|^{2}U_{\mu}(\sigma_{a}(w))dA(w)\right)^{1/2}.\\
&= 1+\lim_{\mathbb{R}\ni a\to1}\left(\int_{\mathbb{D}} \left|\frac{1}{1-\sigma_{a}(w)}\right|^{2}|\sigma_{a}'(w)|^{2}U_{\mu}((w))dA(w)\right)^{1/2}\\
&\geq 1+\left(\int_{\mathbb{D}} \frac{4}{|1-w^{2}|^{2}}U_{\mu}(w)dA(w)\right)^{1/2}.
\end{align*}
\end{proof}
\noindent
A significant application of Theorem \ref{norm h from h to mdm} to $Q_p$ spaces with $\|\cdot \|_{M(\mathcal{D}_{\mu_p})}$ provides the norm of $\mathcal{H}$ in this situation.
\begin{cor}\label{norm H hinf to Qp}
The norm of $\mathcal{H}:H^\infty\to Q_p$ is equal to 
\begin{equation*}\label{Norm Q_p)}
   \|\mathcal{H}\|_{H^\infty\to Q_p}=1+\|\log(1-z)\|_{M(\mathcal{D}_{\mu_p})}.
\end{equation*}  
\end{cor}
\begin{proof}
We apply Theorem \ref{norm h from h to mdm}. We recall that,  if $0<p<1$, then $Q_p=M(\mathcal{D}_{\mu_p})$ where, according to \eqref{mdp},
\[
d\mu_p(z) = -\Delta[(1-|z|^{2})^{p}]dA(z)=4p\left(1-|z|^2 p\right)(1-|z|^2)^{p-2}dA(z). 
\]
\end{proof}

\begin{rem}
 We note that, if we consider $Q_p$ with the norm $\|\cdot\|_{Q_p}$ defined as in \eqref{norm Q_p}, with exactly the same reasoning of Theorem \ref{norm h from h to mdm} we obtain that
 $$
 \|\mathcal{H}\|_{H^\infty\to Q_p}=1+\|\log(1-z)\|_{Q_p}.
 $$
 Indeed $(1-|z|^2)^pdA(z)$ is a radial measure.
\end{rem}

\vspace{22 pt}
\section{The norm of he Cesaro operator}

\noindent We use the following identity 

\begin{align*}
\int_{\mathbb{D}}\left(\int_{0}^{2\pi}|f(e^{i\theta})|^{2}P_{\zeta}(e^{i\theta})\frac{d\theta}{2\pi}-|f(\zeta)|^{2}\right)d\mu(\zeta)
&=\int_{\mathbb{D}} |f'(z)|^{2}U_{\mu}(z)\,dA(z)\\
\end{align*}
where $\mu$ is a positive Borel measure such that $(1-|z|^{2})d\mu(z)$ is finite. 

\begin{proof}[Proof of Theorem \ref{norm C from h to mdm}]

Firstly observe that  $\mathcal{C}(1)(z)=-\frac{1}{z}\log(1-z)$. As for the Hilbert matrix operator, this implies that 
$$
 \|\mathcal{C}\|_{H^\infty\to M(\mathcal{D}_\mu)}\geq1+\left(\int_{\mathbb{D}}\frac{4}{|1-z^{2}|^{2}}U_{\mu}(z)dA(z)\right)^{1/2}.
$$
Let $f\in H^\infty$. Set $h(z)=z\,\mathcal{C}(f)(z)$, then $(1-z)h'(z)=f(z)$. We have that
\begin{align*}
||\mathcal{C}f||_{M(D_\mu)}=|\mathcal{C}(f)(0)|+ \sup_{a\in\mathbb{D}}\left(\int_{\mathbb{D}}|\mathcal{C}(f)'(z)|^2 U_{\mu}(\sigma_{a}(z))dA(z)\right)^{1/2},
\end{align*}
and that $|\mathcal{C}(f)(0)|\leq ||f||_{H^\infty}$. For the integral we work as follows. Let $\mathcal{C}f=g$; then 
\begin{align*}
&\int_{\mathbb{D}}|g'(z)|^2 U_{\mu}(\sigma_{a}(z))dA(z)=\int_{\mathbb{D}}|(g\circ\sigma_{a})'(z)|^2 U_{\mu}(z)dA(z)\\
&=\int_{\mathbb{D}}\left(\int_{0}^{2\pi}|g(\sigma_{a}(e^{i\theta}))|^{2}P_{\zeta}(e^{i\theta})\frac{d\theta}{2\pi}-|g(\sigma_{a}(\zeta))|^{2}\right) d\mu(\zeta)\\
&=\int_{\mathbb{D}}\left(\int_{0}^{2\pi}|h(\sigma_{a}(e^{i\theta}))|^{2}P_{\zeta}(e^{i\theta})\frac{d\theta}{2\pi}-|h(\sigma_{a}(\zeta))|^{2}+|h(\sigma_{a}(\zeta))|^2-|g(\sigma_{a}(\zeta))|^{2}\right) d\mu(\zeta)\\
&=\int_{\mathbb{D}}|h'(z)|^2 U_{\mu}(\sigma_{a}(z))dA(z)-\int_{\mathbb{D}}(1-|\sigma_{a}(\zeta)|^{2})|g(\sigma_{a}(\zeta))|^{2}d\mu(\zeta)\\
&\leq\int_{\mathbb{D}}|h'(z)|^2 U_{\mu}(\sigma_{a}(z))dA(z).
\end{align*}
Since $h'(z)(1-z)=f(z)$ and due to Proposition \ref{norm log MDmu}, we have 
\begin{align*}
||\mathcal{C}f||_{M(D_\mu)}&\leq ||f||_{H^\infty}+ \sup_{a\in\mathbb{D}}\left(\int_{\mathbb{D}}\left|\frac{f(z)}{1-z}\right|^2 U_{\mu}(\sigma_{a}(z))dA(z)\right)^{1/2} \\
&\leq \left(1+\left(\int_{\mathbb{D}}\frac{4}{|1-z^{2}|^{2}}U_{\mu}(z)dA(z)\right)^{1/2}\right)||f||_{H^\infty}.
\end{align*}

\end{proof}

\section{Open questions}
\noindent 
We note that for every radial measure $\mu$ that satisfies \eqref{Def measure},
$$
\|\mathcal{H}\|_{H^\infty \to M(\mathcal{D}_\mu)} =1+\left(\int_{\mathbb{D}}\frac{4}{|1-z^{2}|^{2}}U_{\mu}(z)dA(z)\right)^{1/2}.
$$
However we are not able to compute the exact value of $\|\mathcal{H}\|_{H^\infty\to M(\mathcal{D}_\mu)}$ in general.

\begin{quest}\noindent
What is the exact norm of $\mathcal{H}$ from $H^\infty$ to $M(\mathcal{D}_\mu)$ for non-radial $\mu$?
\end{quest}

\vspace{22 pt}
\section{Acknowledgments}
\noindent 
 The authors would like to express their gratitude to professor Aristomenis Siskakis who point \eqref{Main id for Deriv} out.
\vspace{22 pt}

\bibliographystyle{plain}

\bibliography{ProJectA_B}
\vspace{11 pt}
\end{document}